\documentclass[12pt]{amsart}
\usepackage{amssymb,amsmath,amsthm,amssymb, amsfonts, amscd}
\usepackage{graphicx}
\usepackage{color}
\usepackage[all]{xy}
\usepackage{anysize}
\usepackage[cp1252]{inputenc}
\usepackage[english]{babel}

\newtheorem{theorem}{Theorem}[section]
\newtheorem{lemma}[theorem]{Lemma}
\newtheorem{definition}[theorem]{Definition}
\newtheorem{corollary}[theorem]{Corollary}
\newtheorem{proposition}[theorem]{Proposition}
\newtheorem{remark}[theorem]{Remark}

\DeclareMathAlphabet{\mathpzc}{OT1}{pzc}{m}{it}

\newcommand{\R}{\mathbb{R}}
\newcommand{\N}{\mathbb{N}}
\newcommand{\C}{\mathbb{C}}
\newcommand{\K}{\mathbb{K}}

\def\ov{\overline}

\def\B2star{\overline{B}_X^{w(X^{\ast\ast},X^{\ast})}}

\title{BOUNDED HOLOMORPHIC FUNCTIONS ATTAINING THEIR NORMS IN THE BIDUAL }
\author{Daniel Carando}
\author{Martin Mazzitelli }

\thanks{
This project was partially supported by CONICET PIP 0624, PICT 2011-1456 and UBACyT 20020130100474BA. The
second author is supported by a doctoral fellowship from CONICET}

\address{Departamento de Matem\'{a}tica - Pab I,
Facultad de Cs. Exactas y Naturales, Universidad de Buenos Aires,
(1428) Buenos Aires, Argentina and IMAS-CONICET}

\email{dcarando@dm.uba.ar, mmazzite@dm.uba.ar}

\keywords{Integral formula, norm attaining holomorphic mappings, Lindenstrauss type theorems}
\subjclass[2010] {Primary:  47H60, 46G20. Secondary: 46B28, 46B20}

\date{}

\begin{document}
\baselineskip=.65cm

\begin{abstract}
Under certain hypotheses on the Banach space  $X$, we prove that the set of analytic functions in $\mathcal{A}_u(X)$ (the algebra of all holomorphic and uniformly continuous functions in the ball of $X$) whose Aron-Berner extensions attain their norms, is dense in $\mathcal{A}_u(X)$. This Lindenstrauss type result holds also for functions with values in a dual space or in a Banach space with the so-called property $(\beta)$. We show that the Bishop-Phelps theorem does not hold for $\mathcal{A}_u(c_0,Z'')$ for a certain Banach space $Z$, while our Lindenstrauss theorem does.
In order to obtain our results, we first handle their polynomial cases.
\end{abstract}

\maketitle

\section{Introduction}

The study of the denseness of norm attaining mappings finds its origins in the Bishop-Phelps theorem \cite{BP61}, which asserts that the set of norm attaining bounded linear functionals on a Banach space is norm dense in the space of all bounded and linear functionals. Since the appearance of this result in 1961, the study of norm attaining functions has attracted the attention of many authors.
Given Banach spaces $X$ and $Y$, we say that a linear operator $T\colon X\to Y$  is norm attaining if there exists $x_0$ in the unit ball of $X$ such that $\| T(x_0)\| = \Vert T \Vert$.
A question that arises naturally in this context is if it is possible to generalize the Bishop-Phelps theorem to bounded linear operators. The negative answer was given by
Lindenstrauss in \cite{Lind}. On the other hand, he gave examples of Banach spaces $X$  for which the Bishop-Phelps theorem holds in $\mathcal{L}(X;Y)$ for every Banach space $Y$. Such spaces are said to have property $A$. Similarly, a space $Y$ has property $B$ if the Bishop-Phelps theorem holds in $\mathcal{L}(X;Y)$ for every $X$.
A positive fundamental result given also in \cite{Lind}, the so-called Lindenstrauss theorem for linear operators, states that the set of bounded linear operators (between any two Banach spaces $X$ and $Y$) whose bitransposes are norm attaining, is dense in the space of all operators. This result was generalized by Acosta, García and Maestre in \cite{AGM} for multilinear operators, where the Bishop-Phelps theorem does not hold in general even in the scalar-valued case (we refer the reader to \cite{AAP, JP, CaLaMa12} for counterexamples to the Bishop-Phelps theorem in the multilinear case).
In this context, the role of the bitranspose is played by the canonical (Arens) extension to the bidual, obtained by weak-star density (see \cite{Arens}, \cite[1.9]{DefFlo93} and the definitions below).

In this paper, we study Lindenstrauss type theorems for polynomials and holomorphic functions.
For $2$-homogeneous scalar-valued polynomials, the Lindenstrauss theorem was proved with full generality by Aron, García and Maestre in \cite{ArGM}, where the Aron-Berner extension takes the place of the bitranspose. This result was later extended by Choi, Lee and Song \cite{ChoLeeSon10} for vector-valued $2$-homogeneous polynomials. In \cite{CaLaMa12} a partial result was obtained for homogeneous polynomials of any degree. Specifically,
if $X$, $Y$ are Banach spaces such that $X'$ is separable and has the approximation property, then the set of $N$-homogeneous  polynomials from $X$ to $Y'$ whose Aron-Berner extensions attain their norms is dense in the set of all continuous  $N$-homogeneous polynomials. It is worth noting that in the homogeneous case, there is also no Bishop-Phelps theorem; counterexamples can be found in \cite{JP, CaLaMa12} for, respectively, scalar and vector-valued polynomials.
Going farther, we can ask about the validity of Bishop-Phelps and Lindenstrauss type theorems for non-homogeneous polynomials and holomorphic functions. In this line, our main positive results are the following (see definitions and notation in Section~\ref{seccion dualidad}).

\textbf{Theorem A.} \emph{
Let $X$ be a Banach space whose dual  is separable and has the approximation property. Then, the set of all polynomials of degree at most $k$ from $X$ to $W$ whose Aron-Berner extensions attain their norms is dense in the set of all continuous polynomials of degree at most $k$, whenever $W$ is a dual space or has property $(\beta)$.
}

\textbf{Theorem B.} \emph{
Let $X$ be a Banach space whose dual  is separable and has the approximation property. Then, the set of all functions in $\mathcal{A}_u(X;W)$ whose Aron-Berner extensions attain their norms is dense in $\mathcal{A}_u(X;W)$, whenever $W$ is a dual space or has property $(\beta)$.
}

Theorems A and B are direct consequences of  Theorem~\ref{Lind polinomial no homogeneo}, Corollary~\ref{Lindenstrauss para el algebra uniforme} and Proposition~\ref{Lindenstrauss con propiedad beta}. Indeed, Theorem~\ref{Lind polinomial no homogeneo} and Corollary~\ref{Lindenstrauss para el algebra uniforme} are just the part of  Theorems A and B for mappings with values in dual spaces. In particular, they cover the scalar-valued case. Proposition~\ref{Lindenstrauss con propiedad beta} shows that, if the Lindenstrauss theorem holds in the scalar-valued case, then it also holds with values in a Banach space with property $(\beta)$.

We also deal with stronger versions of Bishop-Phelps and Lindenstrauss theorems. Namely, we consider the density of mappings which attain their suprema in smaller balls, a problem studied, for example, by Acosta, Alaminos, García and Maestre in \cite{AcAlGaMa}. We show in Section 3 that the strong versions of  Theorems A and B hold.

In Section~\ref{seccion contraejemplos al teorema BP} we show that, in general, there are no Bishop-Phelps theorems neither for scalar and vector-valued continuous polynomials (extending some known results) nor for $\mathcal{A}_u(X;Z)$. We remark that for the presented counterexamples, our Lindenstrauss theorem holds. We also address the strong variants of Bishop-Phelps and Lindenstrauss theorems, and show a counterexample of the strong Bishop-Phelps theorem in $\mathcal{A}_u(X)$.

\section{Definitions and preliminary results}\label{seccion dualidad}

Given a Banach space $X$, we denote by $X'$ its dual space, while $B_X$ and $B_X^{º}$ stand, respectively, for the closed and the open unit ball.
By $\mathcal{L}(X_1,\dots,X_N;Y)$ we denote the space of all $N$-linear operators from $X_1\times \dots \times X_N$ to $Y$. This space is endowed with the supremum norm
$$\Vert \Phi\Vert = \sup \{\Vert \Phi(x_1,\dots,x_N)\Vert: \  x_i \in B_{X_i}, \  1\leq i \leq N\}.$$ We say that a multilinear operator $\Phi$ attains its norm if there exists a $N$-tuple $(a_1, \dots, a_N) \in B_{X_1} \times \cdots\times B_{X_N}$ such that $\| \Phi(a_1,\dots,a_N)\| = \| \Phi\|$.

Given $\Phi \in \mathcal{L}(X_1,\dots,X_N;Y)$, its Arens (or canonical) extension is the multilinear operator $\overline{\Phi} : X''_1 \times\cdots\times X''_N \longrightarrow Y''$  defined by
\begin{eqnarray}\label{arens extension}
\overline{\Phi}(x_1'',\ldots,x_N'') = w^* - \lim_{\alpha_1}\ldots\lim_{\alpha_N} \Phi(x_{1,{\alpha_1}},\ldots,x_{N,{\alpha_N}})
\end{eqnarray}
where $(x_{i,{\alpha_i}})_{\alpha_i} \subseteq X$ is a net  $w^*$-convergent to  $x''_i\in X''_i$, $i=1,\ldots, N$.

A continuous \emph{$N$-homogeneous
polynomial} is a function $P\colon X\to Y$ of the form $P(x)= \Phi(x,\ldots,x)$ for some continuous $N$-linear map $\Phi\colon
X\times \overset{N}{\cdots}\times X \to Y$. We denote by ${\mathcal P}(^NX;Y)$ the Banach space of all continuous
$N$-homogeneous polynomials from $X$ to $Y$ endowed with the supremum norm
$$\Vert P\Vert = \sup_{x \in B_X} \Vert P(x)\Vert.$$ Naturally, we say that a polynomial $P$ is \emph{norm attaining} if there exists $x_0 \in B_X$ such that $\Vert P(x_0)\Vert = \Vert P\Vert$. The set of norm attaining $N$-homogeneous polynomials is denoted by ${NA\mathcal P}(^NX;Y)$. We recall that the canonical extension of a polynomial  $P \in \mathcal{P}(^NX;Y)$ to the bidual, usually called the \emph{Aron-Berner extension} \cite{AB}, is the polynomial $\ov P \in \mathcal{P}(^NX'';Y'')$ defined by
$\overline{P}(x'')= \overline{\Phi}(x'',\ldots ,x'')$, where $\Phi$ is the unique symmetric $N$-linear mapping associated to $P$.

%

\medskip

Given $k \in \mathbb N$,  let $\mathcal{P}_k(X;Y)$ denote the Banach space of continuous polynomials from $X$ to $Y$ of degree less than or equal to $k$, endowed with the supremum norm. Each $P \in \mathcal{P}_k(X;Y)$ can be written as $P = \sum_{j=0}^k P_j$, where each $P_j$ is an $j$-homogeneous polynomial.
On the other hand, given a complex Banach space $X$, we denote $\mathcal{A}_u(X;Y)$ to the Banach space of holomorphic functions in the open unit ball $B_X^{º}$ which are uniformly continuous in the closed unit ball $B_X$, endowed with the supremum norm. It is well-known that each $f \in \mathcal{A}_u(X;Y)$ is a uniform limit of polynomials. When $Y=\K$ is the scalar field, we simply write $\mathcal{P}_k(X)$ or  $\mathcal{A}_u(X)$.
As expected, a function $f$ in $\mathcal{P}_k(X;Y)$ or $\mathcal{A}_u(X;Y)$ is said to be norm attaining if there exists $x_0 \in B_X$ such that $\Vert f(x_0)\Vert = \Vert f\Vert$ and the subsets of norm attaining functions are denoted by $NA\mathcal{P}_k(X;Y)$ and $NA\mathcal{A}_u(X;Y)$.
The Aron-Berner extension of a polynomial $P = \sum_{j=0}^k P_j\in \mathcal{P}_k(X;Y)$ is given by $\overline{P} = \sum_{j=0}^k \overline{P_j}$. In the case of a function $f \in \mathcal{A}_u(X;Y)$, given its Taylor series expansion at $0$, $f = \sum_{j=0}^\infty P_j$, the Aron-Berner extension of $f$ is defined as $\overline{f} = \sum_{j=0}^\infty \overline{P_j}$, which is a holomorphic function in the open unit ball $B_{X''}^{°}$ \cite{DG}. Note that, if $(P_n)_{n \in \mathbb N}$ is a sequence of polynomials converging uniformly to $f$, then $(\overline{P_n})_{n \in \mathbb N}$ is uniformly Cauchy in the ball $B_{X''}^{°}$, and then converges uniformly to $\overline{f}$. This means that $\overline{f}$ extends to a uniformly continuous function in the closed unit ball of $X^{''}$.
Davie and Gamelin showed in \cite{DG} that $\|\overline f\|=\|f\|$ in the scalar-valued case. The same holds for a vector-valued $f \in \mathcal{A}_u(X;Y)$, since $\overline{f}(x'')(y') = \overline{y' \circ f}(x'')$ for all $x'' \in X''$ and $y' \in Y'$.

Throughout the article, in the polynomial results the scalar field can be either $ \R$ or $\C$, while we consider only complex Banach spaces in the holomorphic setting.

\subsection*{Duality for non-homogeneous polynomials}

Polynomials in $\mathcal P(^jX)$ can be thought of  as continuous linear functionals on the symmetric projective tensor product as follows.
Given a symmetric tensor $u_j$ in $\otimes^{j,s} X$ (the $j$-fold symmetric tensor product of $X$),  the symmetric projective norm $\pi_s$ of
$u_j$ is defined by
$$
\pi_s(u_j)=\inf\Big\{\sum_{i=1}^m |\lambda_i| \|x_i\|^j\colon u_j=\sum_{i=1}^m\lambda_i x_i^j, (\lambda_i)_{i=1}^m \subset \K, (x_i)_{i=1}^m\subset X\Big\}.
$$
We denote by  $\tilde{\otimes}_{\pi_s}^{j,s}X$ the completion of ${\otimes^{j,s}} X$ with respect to $\pi_s$. Then
$
\mathcal{P}(^jX) = (\tilde{\otimes}_{\pi_s}^{j,s}X)'
$
isometrically, where the identification is given by the duality
$$
L_{p_j}(u_j):=\left\langle u_j,p_j \right\rangle = \sum_{i=1}^\infty \lambda_i p_j(x_i),
$$
for $p_j \in \mathcal{P}(^jX)$ and $u_j \in \tilde{\otimes}_{\pi_s}^{j,s}X$,  $u_j = \sum_{i=1}^\infty \lambda_i x_i^j$.

Consider the space $$G_k = \bigoplus_{j=0}^k (\tilde{\otimes}_{\pi_s}^{j,s}X),$$ where we set $\tilde{\otimes}_{\pi_s}^{0,s}X = \mathbb K$. An element $u \in G_k$ is of the form $u = \sum_{j=0}^k u_j$ with $u_j \in \tilde{\otimes}_{\pi_s}^{j,s}X.$ We endow this space with the norm $$\Vert u \Vert_{G_k} = \sup_{q \in B_{\mathcal{P}_k(X)}} \left\vert \sum_{j=0}^k \langle u_j,q_j \rangle \right\vert,$$where $q_j$ is the $j$-homogeneous part of $q$.
It is easy to check that $(G_k, \Vert \cdot \Vert_{G_k})$ is a Banach space.

With the previous notation, given $p\in \mathcal{P}_k(X)$ we have $\Vert p_j \Vert \leq \Vert p\Vert$ for every $0 \leq j \leq k$ as a consequence of Cauchy inequalities. Therefore, we get for $u_j\in \tilde{\otimes}_{\pi_s}^{j,s}X$
$$\Vert 0+ \cdots + u_j + \cdots + 0\Vert_{G_k} = \sup_{q \in B_{\mathcal{P}_k(X)}} \vert \langle u_j, q_j \rangle \vert \leq \sup_{q \in B_{\mathcal{P}_k(X)}} \Vert u_j\Vert_{\pi_s} \Vert q_j\Vert \leq \Vert u_j\Vert_{\pi_s}.$$
We have shown the following.

\begin{remark}\label{complementados} The space $\mathcal{P}(^jX)$ is 1-complemented in $\mathcal{P}_k(X)$. Also, $\tilde{\otimes}_{\pi_s}^{j,s}X$ is 1-complemented in $G_k$.
\end{remark}

The following lemma shows that $G_k$ \emph{linearizes} polynomials of degree at most $k$.

\begin{lemma} \label{dualidad polinomios no homogeneos}
Let $X$ be a Banach space and $k \in \mathbb N$. The mapping
\begin{eqnarray}
\nonumber \mathcal{P}_k(X) &\longrightarrow& (G_k, \Vert \cdot\Vert_{G_k})' \\
p &\longmapsto& L_p \label{eq-dualidad}
\end{eqnarray}
where $L_p(u) = \langle u,p \rangle = \sum_{j=0}^k \langle u_j,p_j \rangle$, is an isometric isomorphism.
\begin{proof}
Let us see that it is an isometry. By the previous remark,
$$
\vert L_p(u)\vert = \Vert p\Vert \left\vert \sum_{j=0}^k \langle u_j,\frac{p_j}{\Vert p\Vert} \rangle
\right\vert \leq \Vert p\Vert \Vert u\Vert_{G_k},
$$
which implies that $L_p \in G_k'$ with $\Vert L_p\Vert \leq \Vert p\Vert$.
Now, given $\varepsilon >0$ take $x_0 \in B_X$ with $\vert p(x_0)\vert > \Vert p \Vert - \varepsilon$, and consider $$u_0 = \sum_{j=0}^k \underbrace{x_0 \otimes \cdots \otimes x_0}_{j}.$$ Then $u_0 \in B_{G_k}$ and $\vert L_p(u_0)\vert = \vert p(x_0)\vert > \Vert p \Vert - \varepsilon$, which gives the reverse inequality.

Now we prove that the mapping \eqref{eq-dualidad} is surjective. For $L \in G_k'$, let  $L_j$ denote its restriction to  $\tilde{\otimes}_{\pi_s}^{j,s}X$, that is:
\begin{eqnarray*}
L_j := L_{\mid_{\tilde{\otimes}_{\pi_s}^{j,s}X}} : \tilde{\otimes}_{\pi_s}^{j,s}X \longrightarrow \mathbb K \\
L_j(u_j) = L(0+ \cdots + u_j + \cdots + 0).
\end{eqnarray*}
It is clear that $L_j$ is linear and, by Remark~\ref{complementados}, $\vert L(u_j)\vert \leq \Vert L\Vert \Vert u_j\Vert_{\pi_s}$ for each $u_j\in \tilde{\otimes}_{\pi_s}^{j,s}X$. Then $L_j \in (\tilde{\otimes}_{\pi_s}^{j,s}X)'$ and we can find $p_j \in \mathcal{P}(^jX)$ such that $L_j=L_{p_j}$. Now, if we take $p=p_0+\cdots + p_k \in \mathcal{P}_k(X)$ it is easy to check that $L=L_p$.
\end{proof}
\end{lemma}

\bigskip

For polynomials with values in a dual space $Y'$ we have the isometric isomorphism
\begin{equation}\label{dualidad tensores}
\mathcal{P}(^jX;Y') = \left( (\tilde{\otimes}_{\pi_s}^{j,s}X) \tilde{\otimes}_{\pi} Y \right)'.
\end{equation} Here the duality  is given by
\begin{equation}\label{classical duality}
L_{P_j}(u_j):=\left\langle u_j,P_j \right\rangle = \sum_{l=1}^\infty  \sum_{i=1}^\infty \lambda_{l,i} P_j(x_{l,i}) (y_l)
\end{equation}
for any $P_j \in \mathcal{P}(^jX;Y')$ and $u_j= \sum_{l=1}^\infty v_l \otimes y_l$, where $(y_l)_l \subset Y$ and $(v_l)_l \subset \tilde{\otimes}_{\pi_s}^{j,s}X$, with $v_l =\sum_{i=1}^\infty \lambda_{l,i} x_{l,i}^j$ for all $l$.

We define $$G_k = \bigoplus_{j=0}^k \left((\tilde{\otimes}_{\pi_s}^{j,s}X) \tilde{\otimes}_{\pi} Y\right),$$ where the elements are of the form $u = \sum_{j=0}^k u_j$ with $u_j \in (\tilde{\otimes}_{\pi_s}^{j,s}X) \tilde{\otimes}_{\pi} Y$. The norm of such an element is given by $$\Vert u \Vert_{G_k} = \sup_{Q \in B_{\mathcal{P}_k(X; Y')}} \left\vert \sum_{j=0}^k \langle u_j,Q_j \rangle \right\vert.$$ Now the duality $$\mathcal{P}_k(X;Y') \stackrel{1}{=} (G_k, \Vert \cdot \Vert_{G_k})'$$ is defined exactly as in Lemma \ref{dualidad polinomios no homogeneos}, that is, $P \mapsto L_P$ where $L_P(u)=\langle u,P\rangle$.

\bigskip

Note that if we consider the space $$\bigoplus_{j=0}^\infty \left((\tilde{\otimes}_{\pi_s}^{j,s}X) \tilde{\otimes}_{\pi} Y\right)$$ of all the elements $u \in G_k$ for any $k \in \mathbb N \cup \{0\}$, then for each $f = \sum_{j=0}^\infty P_{j} \in H^\infty(B_X^{º};Y')$ and $u = \sum_{j=0}^k u_j$ we have the duality $\langle u,f \rangle = \sum_{j=0}^k \langle u_j,P_{j} \rangle$. We endow this space with the norm $$\Vert u\Vert = \sup_{g \in B_{H^\infty(B_X^{º}; Y')}} \left\vert \langle u,g \rangle\right\vert$$ and we denote its completion by $G_\infty$. An easy calculation shows that the map $f \mapsto L_f$, where $L_f(u)=\langle u,f\rangle$, defines an isometric isomorphism giving the duality $$H^\infty(B_X^{º};Y') \stackrel{1}{=} G_\infty'.$$ We have obtained, in a somehow different way, the space $G_\infty$ constructed by Mujica in~\cite{Mu}. Actually, what we have is a description of this space in terms of tensor products.

\section{An integral formula and the Lindenstrauss type theorems} \label{seccion Lindenstrauss}



In this section we will prove the main results of the article, summarized in Theorems A and B in the Introduction. The following result extends \cite[Theorem 2.2]{CaLaMa12} to the non-homogeneous setting.

\begin{lemma} \label{formula integral polinomios de grado a lo sumo k}
Let $X, Y$ be Banach spaces and suppose that  $X'$ is separable and has the approximation property. Then, for each $u \in G_k$ there exists a regular Borel measure $\mu_u$ on $(B_{X''}, w^*) \times (B_{Y''}, w^*)$ such that $\Vert \mu_u\Vert \leq \Vert u\Vert_{G_k}$ and

\begin{equation} \label{ec. formula integral vectorial}
\left\langle u,P \right\rangle = \int_{B_{X''}\times B_{Y''}} \overline{P}(x'')(y'') d\mu_u(x'',y''),
\end{equation}
for all $P \in \mathcal{P}_k(X;Y')$.
\end{lemma}

\begin{proof}
We first prove the formula for the set $\mathcal{P}_{f,k}(X;Y')$ of finite type polynomials of degree less than or equal to $k$, that is, for the polynomials of the form $P = P_0 + \cdots + P_k$ where the $j$-homogeneous polynomial  $P_j$ is a  linear combination of polynomials of the form $x'(\cdot)^j \cdot y$. Given $u \in G_k$ we define
\begin{eqnarray*}
\Lambda_u : \mathcal{P}_{f,k}(X;Y') \longrightarrow \mathbb C \\
\Lambda_u(P) = \langle u,P\rangle.
\end{eqnarray*}
It is easily verified that $\Vert \Lambda_u\Vert \leq \Vert u\Vert_{G_k}$. Finite type polynomials can be seen as an isometric subspace of $C(B_{X''} \times B_{Y''})$, where the balls are endowed with their weak-star topologies,  identifying a polynomial $P \in \mathcal{P}_{f,k}(X;Y')$ with the function $(x'',y'')\mapsto \overline P(x'')(y'')$. Then, we extend $\Lambda_u$ by the Hahn-Banach theorem to
a continuous linear  functional on $C(B_{X''} \times B_{Y''})$ preserving the norm. Now, by the Riesz representation theorem, there is a regular Borel measure $\mu_u$ on $(B_{X''}, w^*) \times (B_{Y''},w^*)$ such that $\Vert \mu_u\Vert \leq \Vert u\Vert_{G_k}$ and
$$
\Lambda_u (f) = \int_{B_{X''} \times B_{Y''}} f(x'',y'') d\mu_u(x'',y'')
$$
for $f\in C(B_{X''} \times B_{Y''})$, where we still use $\Lambda_u$ for its extension to $C(B_{X''} \times B_{Y''})$. In particular, taking $f=P \in \mathcal{P}_{f,k}(X;Y')$ we obtain the integral formula for finite type polynomials.

Now, take $P = P_0 + \cdots + P_k \in \mathcal{P}_k(X;Y')$. By \cite[Lemma 2.1]{CaLaMa12}, for each $P_j$, $0 \leq j \leq k$, there exists a norm bounded multi-indexed sequence  of  finite type polynomials $(P_{j, n_1,\dots,n_j})_{(n_1,\dots,n_j)\in\N^j}$ satisfying
\begin{equation*} \label{limite de medibles}
\overline{P_j}(x'')(y'') = \lim_{n_1 \rightarrow \infty} \ldots \lim_{n_j \rightarrow \infty} \overline{P_{j, n_1,\ldots,n_j}}(x'')(y'').
\end{equation*}
Fixed $0 \leq j \leq k$ we define $P_{j, n_1,\ldots,n_k} := P_{j, n_1,\ldots,n_j}$ for all $n_{j+1}, \dots, n_k \in \mathbb N$. Then the multi-indexed sequences $(P_{j, n_1,\dots,n_k})_{(n_1,\dots,n_k)\in\N^k}$ are indexed on the same index set and satisfy:
$$\overline{P_j}(x'')(y'') = \lim_{n_1 \rightarrow \infty} \ldots \lim_{n_k \rightarrow \infty} \overline{P_{j, n_1,\ldots,n_k}}(x'')(y'').$$
Now, consider $P_{n_1,\dots,n_k} = \sum_{j=0}^k P_{j, n_1,\ldots,n_k} \in \mathcal{P}_{f,k}(X;Y')$.
Since the integral formula holds for finite type polynomials, we have
\begin{equation*}
\left\langle u,P_{n_1,\ldots,n_k} \right\rangle = \int_{B_{X''} \times B_{Y''}} \overline{P_{n_1,\ldots,n_k}}(x'')(y'') d\mu_u(x'',y''),
\end{equation*}
for all $(n_1,\ldots,n_k) \in \N^k$. As the sequence $(P_{n_1,\dots,n_k})_{(n_1,\dots,n_k)\in\N^k}$ is norm bounded, we may apply $k$-times the bounded convergence theorem to obtain
\begin{eqnarray*}
\lim_{n_1 \rightarrow \infty} \dots \lim_{n_k \rightarrow \infty} \left\langle u,P_{n_1,\ldots,n_k} \right\rangle &=& \lim_{n_1 \rightarrow \infty} \ldots \lim_{n_k \rightarrow \infty} \int_{B_{X''} \times B_{Y''}} \overline{P_{n_1,\ldots,n_k}}(x'')(y'') d\mu_u(x'',y'') \\
&=& \int_{B_{X''} \times B_{Y''}} \overline{P}(x'')(y'') d\mu_u(x'',y'').
\end{eqnarray*}

It remains to show that $\left\langle u,P \right\rangle  = \lim_{n_1 \rightarrow \infty} \ldots \lim_{n_k \rightarrow \infty} \left\langle u,P_{n_1,\ldots,n_k} \right\rangle$. Note that, for each $0\leq j \leq k$, both $\left\langle \ \cdot \ ,P_j \right\rangle$ and $\lim_{n_1 \rightarrow \infty} \ldots \lim_{n_k \rightarrow \infty} \left\langle \  \cdot  \ ,P_{j,n_1,\ldots,n_k} \right\rangle$ are linear continuous functions on $(\tilde{\otimes}_{\pi_s}^{j,s}X) \tilde{\otimes}_{\pi} Y$ which coincide on elementary tensors. Since $P=\sum_{j=0}^k P_j$ and $P_{n_1, \dots,n_k} = \sum_{j=0}^k P_{j,n_1, \dots,n_k}$ the claim follows and the proof is complete.
\end{proof}

Now we are ready to state our Lindenstrauss type theorem for non-homogeneous polynomials. We sketch the proof of the statement which is similar to that of \cite[Theorem~2.3]{CaLaMa12}.

\begin{theorem} \label{Lind polinomial no homogeneo}
Let $X, $ $Y$ be Banach spaces. Suppose that $X'$ is separable and has the approximation property. Then, the set of all polynomials in $\mathcal{P}_k(X;Y')$ whose Aron-Berner extensions attain their norms is dense in $\mathcal{P}_k(X;Y')$.
\end{theorem}

\begin{proof}
Given  $Q \in \mathcal{P}_k(X;Y')$  consider its associated linear function $L_Q \in G_k'$, defined as in  Lemma \ref{dualidad polinomios no homogeneos}.  The Bishop-Phelps theorem asserts that, for $\varepsilon >0$, there exists a norm attaining functional $L=L_P \in G_k'$ such that $\Vert L_Q - L_P \Vert < \varepsilon$, for $P$ some polynomial in $\mathcal{P}_k(X;Y')$.
Since $\Vert L_Q - L_P \Vert = \Vert Q - P \Vert$, it remains to prove that $\overline{P}$ is norm attaining.

Take $u \in  G_k$ such that $\Vert u\Vert_{G_k}=1$ and $\vert L_P(u)\vert = \Vert L_P\Vert = \Vert P\Vert$, and take the regular Borel measure $\mu_{u}$ on $B_{X''} \times B_{Y''}$ given by Lemma~\ref{formula integral polinomios de grado a lo sumo k}.
Then,
\begin{eqnarray*}
\Vert P\Vert = \vert L_P(u)\vert \leq \int_{B_{X''} \times B_{Y''}} \vert\overline{P}(x'')(y'')\vert\ d\vert\mu_{u}\vert(x'',y'') \leq \Vert \overline{P}\Vert \Vert \mu_{u}\Vert \leq \Vert P\Vert.
\end{eqnarray*}
Consequently
$\vert\overline{P}(x'')(y'')\vert = \Vert P\Vert$ almost everywhere (for $\mu_{u}$). Hence $\overline{P}$ attains its norm.
\end{proof}

Since functions in $\mathcal{A}_u(X;Y')$ are uniform limits of polynomials and each polynomial, by the previous theorem, is close to a polynomial whose Aron-Berner extension is norm attaining, we obtain the following Lindenstrauss theorem for the space $\mathcal{A}_u(X;Y')$.

\begin{corollary} \label{Lindenstrauss para el algebra uniforme}
Let $X, $ $Y$ be Banach spaces. Suppose that $X'$ is separable and has the approximation property. Then, the set of all functions in $\mathcal{A}_u(X;Y')$ whose Aron-Berner extensions attain their norms is dense in $\mathcal{A}_u(X;Y')$.
Moreover, given $g \in \mathcal{A}_u(X;Y')$ and $\varepsilon >0$ there exists a polynomial $P$ such that $\overline{P}$ is norm attaining and $\Vert g - P\Vert < \varepsilon$.
\end{corollary}

In order to obtain more examples of spaces on which the Lindenstrauss theorem holds, we bring up the so-called \emph{property} $(\beta)$, which  was introduced by Lindenstrauss in \cite{Lind}, who also showed it to imply property $B$ (see comments in the Introduction). In other words,  if a space $Y$ has property $(\beta)$ then the Bishop-Phelps theorem holds in $\mathcal{L}(X;Y)$ for every Banach space $X$. In the real finite-dimensional case, the spaces with property $(\beta)$ are precisely those whose unit ball is a polyhedron. In the infinite-dimensional case, examples of these spaces are $c_0$, $\ell_\infty$ and {$C(K)$ with $K$ having a dense set of isolated points}. We recall the definition.

\begin{definition}
A Banach space $Y$ has property $(\beta)$ if there exists a subset $\{(y_\alpha, g_\alpha): \, \alpha \in \Lambda\} \subset Y \times Y'$ satisfying:
\begin{enumerate}
\item[$i)$] $\Vert y_\alpha\Vert = \Vert g_\alpha\Vert = g_\alpha(y_\alpha) =1$
\item[$ii)$] There exists $\lambda$, $0\leq \lambda < 1$ such that $\vert g_\alpha(y_\beta)\vert \leq \lambda$ for $\alpha \neq \beta$.
\item[$iii)$] For all $y \in Y$, $\Vert y\Vert = \sup_{\alpha \in \Lambda} \vert g_\alpha(y)\vert$.
\end{enumerate}
\end{definition}

Following the ideas of \cite[Proposition 3]{Lind}, Choi and Kim proved in \cite[Theorem~2.1]{ChoiKim96} that if the Bishop-Phelps theorem holds in $\mathcal{P}(^NX)$, then it holds in $\mathcal{P}(^NX;Y)$ for every space $Y$ with property $(\beta)$. Mimicking their ideas we can prove an analogous statement for the Lindenstrauss theorem, whose proof we give for the sake of completeness. Since there are spaces with property $(\beta)$ which are not dual spaces, this gives new examples of spaces satisfying a Lindenstrauss theorem.

\begin{proposition} \label{Lindenstrauss con propiedad beta}
Suppose that $Y$ has property $(\beta)$. Then, if the Lindenstrauss theorem holds for $\mathcal{P}(^NX)$ (respectively $\mathcal{P}_k(X)$, $\mathcal{A}_u(X)$) then it also holds for $\mathcal{P}(^NX;Y)$ (respectively $\mathcal{P}_k(X;Y)$, $\mathcal{A}_u(X;Y)$).
\end{proposition}

\begin{proof}
We prove the $N$-homogeneous case since the others are completely analogous. Consider $Q \in \mathcal{P}(^NX;Y)$ and $\varepsilon >0$. We may suppose $\Vert Q\Vert =1$ without loss of generality. Note that, since $Y$ has property $(\beta)$, we get easily $1=\Vert Q\Vert = \sup_\alpha \Vert g_\alpha \circ Q\Vert$ and we can take $\alpha_0$ such that $\Vert g_{\alpha_0}\circ Q\Vert \geq 1 - \frac{\varepsilon (1-\lambda)}{4}$. By hypothesis there exists $p \in \mathcal{P}(^NX)$, with $\Vert p\Vert = \Vert g_{\alpha_0}\circ Q\Vert$, such that $\Vert g_{\alpha_0}\circ Q - p\Vert< \frac{\varepsilon (1-\lambda)}{2}$ and $\overline{p}$ attains the norm, say, at $x_0'' \in B_{X''}$. Define $P \in \mathcal{P}(^NX;Y)$ by $P(x) = Q(x) + \left( (1 + \varepsilon) p(x) - g_{\alpha_0}\circ Q(x)\right)y_{\alpha_0}$ and note that
\begin{eqnarray*}
\Vert Q - P\Vert \leq \varepsilon \Vert p\Vert + \Vert g_{\alpha_0}\circ Q -  p \Vert
\leq \varepsilon + \varepsilon (1-\lambda) \leq 2 \varepsilon.
\end{eqnarray*}
It remains to see that $\overline{P}$ is norm attaining.
For this purpose, we need first to prove that $\Vert P\Vert = \Vert g_{\alpha_0}\circ P\Vert$. Note that $\Vert P\Vert = \sup_\alpha \Vert g_{\alpha}\circ P\Vert$ and that given any $\alpha$ we have
\begin{eqnarray*}
\Vert g_{\alpha}\circ P\Vert &\leq& \Vert g_{\alpha}\circ Q\Vert + \vert g_\alpha(y_{\alpha_0})\vert \left( \varepsilon \Vert p\Vert + \Vert p - g_{\alpha_0}\circ Q\Vert\right) \\
&\leq& 1 + \lambda \left(\varepsilon + \frac{\varepsilon (1-\lambda)}{2}\right)
\leq 1+ \frac{\varepsilon(1+\lambda)}{2}.
\end{eqnarray*}
On the other hand, since $g_{\alpha_0}\circ P=(1 + \varepsilon) p$ and $\Vert p\Vert = \Vert g_{\alpha_0}\circ Q\Vert \geq 1 - \frac{\varepsilon (1-\lambda)}{4}$, we have
\begin{eqnarray*}
\Vert g_{\alpha_0}\circ P\Vert \geq (1 + \varepsilon) \left(1 - \frac{\varepsilon (1-\lambda)}{4}\right)
\geq 1+ \frac{\varepsilon(1+\lambda)}{2}
\end{eqnarray*}
which, toghether with the previous inequality, gives $\Vert P\Vert = \Vert g_{\alpha_0}\circ P\Vert$. Noting that $\overline{g_{\alpha_0}\circ P}(x'') = \overline{P}(x'')(g_{\alpha_0})$ and recalling that $\overline{p}$ attains the norm at $x_0''$, we obtain
\begin{eqnarray*}
\Vert P\Vert &=& \Vert g_{\alpha_0}\circ P\Vert = (1 + \varepsilon) \Vert p\Vert = (1 + \varepsilon) \vert \overline{p}(x_0'')\vert \\
&=& \vert \overline{P}(x_0'')(g_{\alpha_0})\vert \leq \Vert \overline{P}(x_0'')\Vert \leq \Vert P\Vert.
\end{eqnarray*}
This proves that $\overline{P}$ is norm attaining, and then the result follows.
\end{proof}

\bigskip

\subsection*{A strong version of the Lindenstrauss theorem}
Up to our knowledge, it is still unknown if the Bishop-Phelps theorem holds for $\mathcal{A}_u(X)$. In \cite{AcAlGaMa}, a different version of the Bishop-Phelps theorem is shown to fail for $\mathcal{A}_u(X)$. Namely, given $0< s \leq 1$ and $f \in \mathcal{A}_u(X)$ we define $$\Vert f \Vert_s = \sup\{|f (x)|: \Vert x\Vert \leq s\}$$ which is clearly a norm on $\mathcal{A}_u(X)$; note that for $s=1$ we get the usual supremum norm denoted by $\Vert\cdot\Vert$. Then, we can ask about the denseness of functions that attain the $\Vert \cdot\Vert_s$-norm. Note that given $0 < s \leq s_0 \leq 1$, if the $\Vert \cdot\Vert_s$-norm attaining functions are $\Vert \cdot\Vert_{s_0}$-dense (that is, dense when considering the $\Vert \cdot\Vert_{s_0}$-norm) in $\mathcal{A}_u(X)$, then the Bishop-Phelps theorem holds. Indeed, given $g \in \mathcal{A}_u(X)$ and $\varepsilon >0$ take a polynomial $q$ such that $\Vert g - q\Vert< \varepsilon/2$ and consider $q_{_\frac{1}{s}}$ defined by $q_{_\frac{1}{s}}(\cdot)=q(\frac{1}{s} \cdot)$. By the assumption, we have a $\Vert \cdot\Vert_s$-norm attaining function $f \in \mathcal{A}_u(X)$ such that $\Vert q_{_\frac{1}{s}} - f\Vert_{s_0}< \varepsilon/2$. If we define $f_s \in \mathcal{A}_u(X)$ by $f_s(\cdot)=f(s \cdot)$, then $f_s$ is $\Vert \cdot\Vert$-norm attaining and $\Vert f_s\Vert=\Vert f\Vert_s$. On the other hand, $\Vert q - f_s\Vert = \Vert q_{_\frac{1}{s}} - f\Vert_s \leq \Vert q_{_\frac{1}{s}} - f\Vert_{s_0} < \varepsilon/2$ and consequently $\Vert g - f_s\Vert < \varepsilon$. The same holds in the vector-valued case.

We will refer to these type of results (i.e., the denseness of functions that attain the $\Vert \cdot\Vert_s$-norm) as \emph{strong versions of the Bishop-Phelps theorem}. When these stronger versions come into scene, we will specify carefully whether we consider the $\Vert\cdot\Vert$-norm or some $\Vert\cdot\Vert_s$-norm; otherwise, the usual supremum norm is taken without considerations.
The following result will be improved in Section \ref{seccion contraejemplos al teorema BP}, where also the definition of the preduals of Lorentz sequence spaces will be given.

\begin{theorem} \cite[Corollary~4.5]{AcAlGaMa}
Let $X = d_*(w,1)$ with $w \in \ell_2 \backslash \ell_1$. Given $0 < s <1/e$, the set of elements of $\mathcal{A}_u(X)$ that attain the $\Vert \cdot\Vert_s$-norm
is not $\Vert \cdot\Vert$-dense in $\mathcal{A}_u(X)$.
\end{theorem}

Taking  this result into account, it is natural to ask if a Lindenstrauss theorem holds for the $\Vert \cdot \Vert_s$-norm in $\mathcal{A}_u(X;Y')$. Our goal now is to give a partial positive answer to this problem. We briefly sketch the arguments, since they are slight modifications of those followed in the first part of this section.
First, we state the following more general version of the well-known Bishop-Phelps theorem (see \cite{BP63} or the final comment added in \cite{BP61}).

\medskip
$(\star)$ \hspace{1cm} \textit{Let $X$ be a real Banach space, $C \subseteq X$ a bounded closed convex set and $$C^*= \{\varphi \in X' : \,\, \varphi(x_0) = \sup_{x \in C} \varphi(x), \,\,\, \text{for some $x_0 \in C$}\}.$$ Then $C^*$ is dense in $X'$. If in addition $C$ is balanced, then for $X$ real or complex Banach space the set $C^*= \{\varphi \in X' : \,\, \vert\varphi(x_0)\vert = \sup_{x \in C} \vert\varphi(x)\vert, \,\,\, \text{for some $x_0 \in C$}\}$ is dense in $X'$.}
\medskip

\noindent Given $X$, $Y$ Banach spaces and $0<s\leq 1$, recall $G_k$ the predual of $\mathcal{P}_k(X;Y')$ defined in Section \ref{seccion dualidad} and consider the subset $$C_s = \left\{u \in G_k: \,\, \sup_{\Vert Q\Vert_s \leq 1} \vert \langle u,Q\rangle\vert \leq 1\right\},$$ which turns to be a bounded, closed, balanced and convex set.
It is easily verified that $\sup_{u \in C_s} \vert L_P(u)\vert = \Vert P\Vert_s$ for any $P \in \mathcal{P}_{k}(X;Y')$. Also, if we take $P_s(\cdot) = P(s \cdot)$, it can be checked that $(\overline{P})_s = \overline{P_s}$ and $\|\overline P\|_s=\|P\|_s$.
For elements in $C_s$, we have the following generalization of the integral formula presented in Lemma \ref{formula integral polinomios de grado a lo sumo k}. The proof is analogous, just taking the $\Vert \cdot\Vert_s$-norm in the subspace of finite type polynomials instead of the usual supremum norm.

\begin{lemma}\label{formula integral strong}
Let $0<s \leq1$ and $X, Y$ Banach spaces such that  $X'$ is separable and has the approximation property. Then, for each $u \in C_s$ there exists a regular Borel measure $\mu_u$ on $(s B_{X''}, w^*) \times (B_{Y''}, w^*)$ such that $\Vert \mu_u\Vert \leq 1$ and

\begin{equation} \label{ec. formula integral vectorial 2}
\left\langle u,P \right\rangle = \int_{s B_{X''}\times B_{Y''}} \overline{P}(x'')(y'') d\mu_u(x'',y''),
\end{equation}
for all $P \in \mathcal{P}_k(X;Y')$.
\end{lemma}

We state now our strong version of the Lindenstrauss theorem, generalizing the Lindenstrauss type results obtained in Corollary \ref{Lindenstrauss para el algebra uniforme} and Proposition \ref{Lindenstrauss con propiedad beta}. The proof is analogous to the corresponding results for the supremum norm, making use  of the more general versions of the Bishop-Phelps theorem $(\star)$ and the integral formula (\ref{ec. formula integral vectorial 2}) of Lemma~\ref{formula integral strong}.
\begin{theorem}\label{strong Lindenstrauss}
Let  $0<s \leq 1$ and suppose that $X'$ is separable and has the approximation property. Then, for $W$  either a dual space or a Banach space with property $(\beta)$, the set of polynomials from $X$ to $W$ whose Aron-Berner extensions attain their $\Vert \cdot\Vert_s$-norms is $\Vert \cdot\Vert$-dense in $\mathcal{A}_u(X;W)$.
\end{theorem}

If $0<s\leq s_0 \leq 1$, $W$ is a dual space or has property $(\beta)$ and for $X'$ under the usual hypothesis of separability and approximation property, the previous theorem trivially implies the $\Vert \cdot\Vert_{s_0}$-denseness in $\mathcal{A}_u(X;W)$ of the polynomials whose Aron-Berner extensions attain their $\Vert \cdot\Vert_s$-norms. In particular, the set of polynomials whose Aron-Berner extensions attain their $\Vert \cdot\Vert_s$-norms is $\Vert \cdot\Vert_s$-dense in $\mathcal{A}_u(X;W)$. This last \textit{strong version} of Corollary \ref{Lindenstrauss para el algebra uniforme} and Proposition \ref{Lindenstrauss con propiedad beta} is actually equivalent to that one. Indeed, take $g \in \mathcal{A}_u(X;W)$  and $\varepsilon >0$. Consider $g_s \in \mathcal{A}_u(X;W)$ defined by $g_s(\cdot)=g(s \cdot)$. By the assumption there exists a polynomial $P$ such that $\overline{P}$ is $\Vert \cdot\Vert$-norm attaining and $\Vert g_s - P\Vert < \varepsilon$. Take $P_{_\frac{1}{s}}(\cdot) = P(\frac{1}{s} \cdot)$ and note that $\Vert P_{_\frac{1}{s}}\Vert_s = \Vert P\Vert$ and $\overline{P_{_\frac{1}{s}}}$ is $\Vert \cdot\Vert_s$-norm attaining. On the other hand, it is easy to see that $\Vert g - P_{_\frac{1}{s}}\Vert_s = \Vert g_s - P\Vert < \varepsilon$. 

Note that for $g\in H^\infty(B_X^{º}; W)$ and $0<s_0<1$, the function $g_{s_0}(\cdot)=g(s_0\cdot)$ belongs to $\mathcal{A}_u(X;W)$. As a consequence of the previous theorem, given $0<s\leq s_0<1$, if $X'$ is separable and has the approximation property and $W$ is a dual space or has property $(\beta)$ then the set of polynomials whose Aron-Berner extensions attain their $\Vert \cdot\Vert_s$-norms is $\Vert \cdot\Vert_{s_0}$-dense in $H^\infty(B_X^{º}; W)$. We do not know wether the same is true for $s_0=1$.

\section{Counterexamples to Bishop-Phelps theorems} \label{seccion contraejemplos al teorema BP}


The preduals of Lorentz sequence spaces appear related to the study of denseness of norm attaining functions as a useful tool in finding counterexamples to the Bishop-Phelps type theorems. It was Gowers in \cite{G90} the first to consider such a  predual  to prove that the spaces $\ell_p$ ($1 < p<\infty$) do not have the property $B$ of Lindenstrauss. Later,  the same space was used in \cite{AAP} to show the failure of the Bishop-Phelps theorem for bilinear forms and 2-homogeneous scalar-valued polynomials. In \cite{JP}, the authors characterize those preduals of Lorentz sequence spaces in which Bishop-Phelps theorem holds for multilinear forms and $N$-homogeneous scalar-valued polynomials.

We recall now some definitions and properties (for further details on Lorentz sequence spaces, see \cite[Chapter~4.e]{LT1}). An \emph{admissible sequence} will mean a decreasing sequence $w=(w_i)_{i \in \mathbb{N}}$ of nonnegative real numbers with $w_1 = 1$, $\lim w_i =0$ and $\sum_{i}w_i = \infty$. The real or complex Lorentz sequence space $d(w,1)$ associated to an admissible sequence $w=(w_i)_{i\in \N}$ is the vector space of all bounded sequences $x=(x(i))_i$ such that
$$
\Vert x\Vert_{w,1} : =\sum_{i=1}^{\infty} x^*(i) w_i < \infty,
$$
where $x^*=(x^*(i))_i$ is the decreasing rearrangement of $(x(i))_i$.  This is a nonreflexive Banach space when is endowed with the norm $\Vert \cdot\Vert_{w,1}$.
It is known that the predual of the Lorentz space $d(w,1)$, which is denoted by $d_*(w,1)$, is the space of all the sequences $x$ such that
$$
\lim_{n\to \infty}\frac{\sum_{i=1}^n x^*(i)}{W(n)}=0
$$
where $W(n)= \sum_{i=1}^n w_i$. In this space the norm is defined by
$$
\Vert x\Vert_W : = \sup_{n}\frac{\sum_{i=1}^n x^*(i)}{W(n)} < \infty.
$$
Note that the condition $w_1=1$  is equivalent to the assumption that $\|e_i\|_W=1$ for all $i$ in $\N$, where $e_i$ stands for the canonical $i$-th vector of $d_*(w,1)$.

There are two fundamental properties of the spaces $d_*(w,1)$, which make them important in the study of these topics. The first one is related to the geometry of the unit ball, more precisely with the lack of extreme points. The second is about the inclusion of these spaces on $\ell_r$ whenever the admissible sequence $w$ belongs to $\ell_r$. We state these properties whose demonstrations can be found, for instance, in \cite[Lemma~2.2 and Proposition 2.4]{JP}.
\begin{itemize}
\item Given $x \in B_{d_*(w,1)}$, there exists $n_0 \in \mathbb{N}$ and $\delta >0$ such that
$\Vert x + \lambda e_n\Vert_W \leq 1$, for all $\ |\lambda| \le  \delta$ and $n \geq n_0$.
\item If $w \in \ell_r$, $1 < r< \infty$, then the formal inclusion $d_*(w,1) \hookrightarrow \ell_r$ is bounded.
\end{itemize}
It is important to mention that preduals of Lorentz sequence spaces have shrinking basis and, hence, satisfy the hypothesis of the Lindenstrauss type theorems proved in Section~\ref{seccion Lindenstrauss}.
From now on, $w$ will denote an admissible sequence.

\subsection*{Counterexamples in the polynomial case}

Let us summarize in the following auxiliary lemma some known results about bounds on the derivatives of polynomials; see, for instance, \cite{Ha75}, \cite{Ha02}.

\begin{lemma} \label{derivadas acotadas}
Let $X$ and $Y$ be Banach spaces over the scalar field $\mathbb K = \mathbb R$ or $\mathbb C$. Fixed $1 \leq j\leq k$ natural numbers, there exist a constant $C_{k,j} >0$ (depending only on $j$ and $k$) such that
\begin{equation*}
\left\Vert \frac{D^j P(x)}{j!} \right\Vert \leq C_{k,j} \Vert P\Vert
\end{equation*}
for every $P \in \mathcal{P}_k(X;Y)$ and $x \in B_X$.
\end{lemma}

The following results extend Lemma~3.1 and Theorem~3.2 of \cite{JP} to the non-homogeneous case.

\begin{lemma} \label{derivadas igual a 0 caso polinomial}
Let $X$ be a complex Banach sequence space and $W$ be strictly convex. Suppose that a polynomial $P : X \rightarrow W$ attains the norm at an element $x_0 \in  B_X$ satisfying the following condition:
\begin{equation} \label{propiedad extremal 1}
\exists \  n_0 \in \mathbb{N} \  \text{\ and\ }\ \delta >0\  \text{\ such that\ }\quad  \Vert x_0 + \lambda e_n\Vert \leq 1,\quad  \forall \ |\lambda| \le  \delta \text{\  and \  }  n \geq n_0.
\end{equation}
Then, $D^j P(x_0)(e_n) =0$ for all $j\geq 1$ and $n \geq n_0$.
\end{lemma}

\begin{proof}
Fix $n \geq n_0$. Since $P$ attains the norm at $x_0$,  the modulus of the one variable holomorphic  function
\begin{eqnarray*}
\{\vert\lambda\vert < \delta\} &\longrightarrow& W \\
\lambda &\mapsto& P(x_0 + \lambda e_n)
\end{eqnarray*}
attains a local maximum at the origin. By the maximum modulus principle, this function must be constant. Let us see that this implies that $D^j P(x_0)(e_n) =0$ for all $j\geq 1$. Consider the series expansion of $P$ at $x_0$, $$P(x) = \sum_{j=0}^\infty \frac{D^j P(x_0)}{j!}(x-x_0).$$ Evaluating in $x=x_0 + \lambda e_n$ and recalling that $\lambda \mapsto P(x_0 + \lambda e_n)$ is a constant function we obtain $$P(x_0) = P(x_0 + \lambda e_n) = P(x_0) + \sum_{j=1}^\infty \frac{D^j P(x_0)}{j!}(e_n) \lambda^j$$ for all $\vert \lambda\vert < \delta$. Then $0= \sum_{j=1}^\infty \frac{D^j P(x_0)}{j!}(e_n) \lambda^j$ for all $\vert \lambda\vert < \delta$ and consequently $D^j P(x_0)(e_n) =0$ for all $j\geq 1$.
\end{proof}

\begin{proposition} \label{equivalencias caso polinomial}
Given an admissible sequence $w$ and $N \geq 2$, the following statements are equivalent.

\begin{enumerate}
\item[\rm (i)] $NA\mathcal{P}(^Nd_*(w,1))$ is dense in $\mathcal{P}(^Nd_*(w,1))$.

\item[\rm (ii)] If $k\geq N$, every $N$-homogeneous polynomial in $\mathcal{P}(^Nd_*(w,1))$ can be approximated by norm attaining polynomials in $\mathcal{P}_k(d_*(w,1))$.

\item[\rm (iii)] $w \notin \ell_N$.
\end{enumerate}
\end{proposition}

\begin{proof}
The implication $(i) \Rightarrow (ii)$ is trivial, while $(iii) \Rightarrow (i)$ follows from \cite[Theorem 3.2]{JP}.
Let us show that $(ii) \Rightarrow (iii)$. We suppose that $w \in \ell_N$ and give different proofs for $d_*(w,1)$ complex or real Banach space.

\textit{The complex case}. Take $q \in \mathcal{P}(^Nd_*(w,1))$ defined by $q(x)= \sum_{i=1}^\infty
x(i)^N$ (here we use that $d_*(w,1) \hookrightarrow \ell_N$). If $p$ attains its norm at some $x_0 \in
B_{d_*(w,1)}$, then the Lemma \ref{derivadas igual a 0 caso polinomial} assures that $D^N p(x_0)(e_n) =0$ for
$n$ sufficiently large. Since $\frac{D^N q(x_0)}{N!} = q$, by Lemma \ref{derivadas acotadas} we obtain for
large $n$,
$$1=\left\vert \frac{D^N q(x_0)}{N!}(e_n) - \frac{D^N p(x_0)}{N!}(e_n)\right\vert \leq \left\Vert \frac{D^N q(x_0)}{N!} - \frac{D^N p(x_0)}{N!}\right\Vert \leq C_{k,N} \Vert q - p\Vert.$$
Therefore, $q$ cannot be approximated by norm attaining polynomials.

\textit{The real case}. Let $M \leq N$ be the smallest natural number such that $w \in \ell_M$, consider $q \in \mathcal{P}(^Nd_*(w,1))$ defined by $q(x) = x(1)^{N-M} \sum_{i=1}^\infty (-1)^i x(i)^M$ and suppose that $q$ is approximated by norm attaining polynomials in $\mathcal{P}_k(d_*(w,1))$. Fix $\varepsilon >0$ and, in virtue of Lemma \ref{derivadas acotadas}, take $p \in NA\mathcal{P}_k(d_*(w,1))$ such that
\begin{equation} \label{p aproxima q}
\left\Vert \frac{D^M q(x)}{M!} - \frac{D^M p(x)}{M!}\right\Vert < \varepsilon
\end{equation}
for any $x \in B_{d_*(w,1)}$. Now, let $x_0 \in B_{d_*(w,1)}$ be such that $\Vert p\Vert = \vert p(x_0)\vert$ and take $n_0 \in \mathbb N$ and $\delta >0$ so that (\ref{propiedad extremal 1}) is satisfied. Assume for the moment that $p(x_0) >0$. Then we have
\begin{equation} \label{desigualdad caso real}
p(x_0 + \lambda e_n) = p(x_0) + \sum_{j=1}^k \frac{D^j p(x_0)}{j!}(e_n) \lambda^j \leq p(x_0)
\end{equation}
for $\vert \lambda \vert < \delta$ and $n \geq n_0$. By \cite[Theorem 3.2]{JP}, for any $j < M$ the $j$-homogeneous polynomial $\frac{D^j p(x_0)}{j!}$ is weakly sequentially continuous and consequently $\lim_{n \rightarrow \infty} \frac{D^j p(x_0)}{j!}(e_n) =0$. Then, taking limits in (\ref{desigualdad caso real}) and dividing by $\lambda^M$ we obtain,
$$
\limsup_{n\to \infty} \sum_{j=M}^k \frac{D^j p(x_0)}{j!}(e_n) \lambda^{j-M} \leq 0
$$
for $0\leq \lambda < \delta$. As a consequence, $$\limsup_{n\to \infty} \frac{D^M p(x_0)}{M!}(e_n) \leq 0.$$ If $p(x_0) <0$, reasoning with $-p$ we get $$\liminf_{n\to \infty} \frac{D^M p(x_0)}{M!}(e_n) \geq 0.$$
On the other hand, an easy calculation shows that $$\limsup_{n\to \infty} \frac{D^M q(x_0)}{M!}(e_n) = \vert x_0(1)\vert^{N-M} = - \liminf_{n\to \infty} \frac{D^M q(x_0)}{M!}(e_n).$$
Now, by (\ref{p aproxima q}), if $p(x_0)>0$ we have
$$
\vert x_0(1)\vert^{N-M} = \limsup_{n\to \infty} \frac{D^M q(x_0)}{M!}(e_n) \leq \limsup_{n\to \infty} \frac{D^M p(x_0)}{M!}(e_n) + \varepsilon \leq \varepsilon,
$$
while if $p(x_0)<0$ then
$$
-\vert x_0(1)\vert^{N-M} = \liminf_{n\to \infty} \frac{D^M q(x_0)}{M!}(e_n) \geq \liminf_{n\to \infty} \frac{D^M p(x_0)}{M!}(e_n) - \varepsilon \geq -\varepsilon.
$$
Therefore,
\begin{eqnarray*}
\Vert q\Vert &\leq & \Vert p\Vert + \varepsilon = \vert p(x_0)\vert + \varepsilon \leq \vert q(x_0)\vert + 2 \varepsilon \\
&\leq& \vert x_0(1)\vert^{N-M} \left( \sum_{i=1}^\infty \vert x_0(i)\vert^M\right) + 2 \varepsilon <
\varepsilon \left( \sum_{i=1}^\infty w(i)^M + 2\right).
\end{eqnarray*}
Since $\varepsilon$ was arbitrary we have $\Vert q \Vert =0$, which is the desired contradiction.
\end{proof}

The following corollary improves \cite[Corollary~4.4]{AcAlGaMa}.

\begin{corollary} \label{corolario equivalencias caso polinomial}
$NA\mathcal{P}_N(d_*(w,1))$ is dense in $\mathcal{P}_N(d_*(w,1))$ if and only if $w\notin \ell_N$.
\end{corollary}

\begin{proof}
It suffices to prove that $w \notin \ell_N$ implies $NA\mathcal{P}_N(d_*(w,1))$ is dense in $\mathcal{P}_N(d_*(w,1))$; the other implication follows from the previous proposition. Note that if $w \notin \ell_N$ then $w \notin \ell_j$ for all $j\leq N$. As a consequence of \cite[Theorem~3.2]{JP}, we have that $\mathcal{P}(^jd_*(w,1))=\mathcal{P}_{wsc}(^jd_*(w,1))$ for all $j\leq N$ and then $\mathcal{P}_N(d_*(w,1))=\mathcal{P}_{N,wsc}(d_*(w,1))$. Now, following the arguments in the proof of \cite[Theorem~3.2]{JP} we obtain the desired result.
\end{proof}

Finally, we give some counterexamples in the vector-valued case.

\begin{proposition}
Let $w$ be an admissible sequence and $N \geq 2$.
\begin{itemize}
\item[\rm (i)] Suppose $d_*(w,1)$ is a complex Banach space and $W$ is strictly convex.
\begin{enumerate}
\item[\rm a)] If $w \in \ell_N$, then the set $NA\mathcal{P}_k(d_*(w,1);W)$ is not dense in $\mathcal{P}_k(d_*(w,1);W)$ for any $k\geq N$.
\item[\rm b)] If $w \in \ell_r$ for some $1 < r <\infty$, then the set $NA\mathcal{P}_k(d_*(w,1);\ell_r)$ is not dense in $\mathcal{P}_k(d_*(w,1);\ell_r)$ for any $k\geq 1$.
\end{enumerate}
\item[\rm (ii)] If $d_*(w,1)$ is a real Banach space and $w \in \ell_M \backslash \ell_{M-1}$ for some $M$
even, then the set $NA\mathcal{P}_k(d_*(w,1);\ell_M)$ is not dense in $\mathcal{P}_k(d_*(w,1);\ell_M)$ for any
$k \geq 1$.
\end{itemize}
\end{proposition}

\begin{proof}
$(i)$ For $a)$, since $w \in \ell_N$, fixed a norm-one element $z_0\in W$ we can define $Q \in \mathcal{P}(^Nd_*(w,1);W)$  by $$Q(x)=\left(\sum_{i=1}^\infty x(i)^N\right) z_0 .$$
Now the proof follows exactly as in Proposition \ref{equivalencias caso polinomial}. For item $b)$ take $Q(x)=x$, which clearly belongs to $\mathcal{P}_k(d_*(w,1);\ell_r)$ for every $k \geq 1$, and reason again as in Proposition \ref{equivalencias caso polinomial}.

$(ii)$ Consider $Q(x)=x$ and suppose that is approximated by norm attaining polynomials in $\mathcal{P}_k(d_*(w,1); \ell_M)$. Since norm-one $M$-homogeneous polynomials are uniformly equicontinuous, given $\varepsilon >0$ we can take $P \in NA\mathcal{P}_k(d_*(w,1);\ell_M)$ so that $$\Vert q \circ Q - q \circ P\Vert < \varepsilon  C_{Mk,M}^{-1}$$ for every norm-one polynomial $q \in \mathcal{P}(^M \ell_M)$, where $C_{Mk,M}$ is the constant given in Lemma \ref{derivadas acotadas}.

Now if $x_0 \in B_{d_*(w,1)}$ is such that $\Vert P(x_0)\Vert = \Vert P\Vert$, we consider the norm-one $M$-homogeneous polynomial $q_{P,x_0} : \ell_M \longrightarrow \mathbb R$ given by $q_{P,x_0}(x) = \sum_{i=1}^\infty x(i)^M$. Note that $q_{P,x_0} \circ P \in \mathcal{P}_{Mk}(d_*(w,1))$ is norm attaining and $q_{P,x_0} \circ P(x_0) = \Vert P\Vert^M$. On the other hand, $q_{P,x_0} \circ Q(x) = \sum_{i=1}^\infty x(i)^M$ and by the previous inequality we have $$\left\Vert \frac{D^M (q_{P,x_0} \circ Q)(x)}{M!} - \frac{D^M (q_{P,x_0} \circ P)(x)}{M!}\right\Vert < \varepsilon.$$
Reasoning as in Proposition \ref{equivalencias caso polinomial} we get $$\limsup_{n\to \infty} \frac{D^M (q_{P,x_0} \circ P)(x_0)}{M!}(e_n) \leq 0 \quad \text{and} \quad \limsup_{n\to \infty}\frac{D^M (q_{P,x_0} \circ Q)(x_0)}{M!}(e_n)=1.$$ Hence $$1 = \limsup_{n\to \infty}\frac{D^M (q_{P,x_0} \circ Q)(x_0)}{M!}(e_n) \leq \limsup_{n\to \infty}\frac{D^M (q_{P,x_0} \circ P)(x_0)}{M!}(e_n) + \varepsilon \leq \varepsilon,$$ and since $\varepsilon$ was arbitrary, we obtain the desired contradiction.

\end{proof}

In view of the Proposition \ref{Lindenstrauss con propiedad beta}, it is interesting to find counterexamples to the Bishop-Phelps theorem when the polynomials take values on spaces with property $(\beta)$.

\begin{proposition} \label{equivalencias caso polinomial a valores en c0}
Let $w$ be an admissible sequence and $N \geq 2$.
\begin{enumerate}
\item[\rm (i)] $NA\mathcal{P}(^Nd_*(w,1); c_0)$ is dense in $\mathcal{P}(^Nd_*(w,1); c_0)$ if and only if $w\notin \ell_N$.
\item[\rm (ii)] $NA\mathcal{P}_N(d_*(w,1); c_0)$ is dense in $\mathcal{P}_N(d_*(w,1); c_0)$ if and only if $w\notin \ell_N$.
\end{enumerate}
\end{proposition}

\begin{proof}
$(i)$ If $w \notin \ell_N$ then $NA\mathcal{P}(^Nd_*(w,1))$ is dense in $\mathcal{P}(^Nd_*(w,1))$ and, since $c_0$ has property $(\beta)$, by \cite[Theorem 2.1]{ChoiKim96} we obtain that $NA\mathcal{P}(^Nd_*(w,1); c_0)$ is dense in $\mathcal{P}(^Nd_*(w,1); c_0)$.

For the other implication suppose that $w \in \ell_N$ and take $Q: d_*(w,1) \rightarrow c_0$ defined by $$Q(x) = \left(\sum_{i=1}^\infty x(i)^N, \sum_{i=2}^\infty x(i)^N, \sum_{i=3}^\infty x(i)^N, \dots \right)$$ in the complex case and by $$Q(x) = x(1)^{N-M}\left(\sum_{i=1}^\infty (-1)^i x(i)^M, \sum_{i=2}^\infty (-1)^i x(i)^M, \sum_{i=3}^\infty (-1)^i x(i)^M, \dots \right)$$ in the real case, where $M$ is the smallest natural number such that $w \in \ell_M$.
Suppose that $Q$ is approximated by norm attaining polynomials in $\mathcal{P}(^Nd_*(w,1); c_0)$ and take $P \in NA\mathcal{P}(^Nd_*(w,1); c_0)$. Let us see that there exists $m_0$ such that $e_{m_0}^* \circ P$ is norm attaining. Indeed, let $x_0 \in B_{d_*(w,1)}$ be such that $$\Vert P(x_0)\Vert = \sup_n \vert e_{n}^* \circ P(x_0)\vert = \Vert P\Vert.$$ Since $P(x_0) \in c_0$ the supremum in the last equality is actually a maximum and consequently there exists $m_0$ such that $\vert e_{m_0}^* \circ P(x_0)\vert = \Vert P\Vert = \Vert e_{m_0}^* \circ P\Vert$. Noting that $\Vert  e_{m_0}^* \circ Q- e_{m_0}^* \circ P\Vert \leq \Vert Q-P\Vert$ and reasoning as in \cite[Theorem 3.2]{JP}, we get the desired contradiction.

$(ii)$ The proof is analogous, but reasoning as in Proposition \ref{equivalencias caso polinomial} instead of \cite[Theorem 3.2]{JP}.
\end{proof}

\subsection*{Counterexample in $\mathcal{A}_u$}

Let us see  that the Bishop-Phelps theorem does not hold for $\mathcal{A}_u$ in the vector-valued case. First we need the following auxiliary lemma. Recall that, as we already mention in Section \ref{seccion dualidad}, in the holomorphic results we consider only complex Banach spaces.

\begin{lemma} \label{lema extremal 2}
Let $X$ be a Banach sequence space and $W$ be strictly convex. Suppose $x_0 \in B_X$ satisfies the following condition:
\begin{equation}\label{propiedad extremal}
\exists \  n_0 \in \mathbb{N} \  \text{\ and\ }\ \delta >0\  \text{\ such that\ }\quad  \Vert x_0 + \lambda e_n\Vert \leq 1,\quad  \forall \ |\lambda| \le  \delta \text{\  and \  }  n \geq n_0.
\end{equation}
Then, for any  $f \in \mathcal{A}_u(X;W)$ and any $n \geq n_0$, the function
\begin{eqnarray*}
g_f: \left\{\vert \lambda\vert <\delta/2\right\} & \longrightarrow & W \\
\lambda & \longmapsto & f(x_0 + \lambda e_n)
\end{eqnarray*}
is holomorphic.
\end{lemma}

We remark that every element in the unit ball of $c_0$ or $d_*(w,1)$ satisfies condition~\eqref{propiedad extremal}.

\begin{proof}
Take a sequence $(\alpha_i)_{i\in \mathbb N} \subset \mathbb R$ such that $1/2 < \alpha_i < 1$ and $\alpha_i \nearrow 1$. For $n \geq n_0$ we define $g_i : \{\vert \lambda\vert <\delta/2\} \rightarrow W$  by $$g_i(\lambda) = f(\alpha_i x_0 + \lambda e_n).$$ Since $\alpha_i x_0 + \lambda e_n $ belongs to $ \alpha_i B_X$ for all $\vert \lambda\vert < \delta/2$, the function $g_i$ is holomorphic for all $i \geq 1$. Let us show that $g_i$ converges uniformly to $g_f$. Since $f$ is uniformly continuous, given $\varepsilon >0$ there exists $\delta' >0$ such that $\Vert f(x) - f(y)\Vert < \varepsilon$ whenever $x, y \in B_X$ satisfy $\Vert x-y\Vert < \delta'$. Taking $i$ sufficiently large, we have $1 - \alpha_i < \delta'$ and consequently $\Vert (\alpha_i x_0 + \lambda e_n) - (x_0 + \lambda e_n)\Vert < \delta'$. Then, there exists $i_0$ such that $$\Vert g_i(\lambda) - g_f(\lambda)\Vert = \Vert f(\alpha_i x_0 + \lambda e_n) - f(x_0 + \lambda e_n) \Vert < \varepsilon$$ for all $\vert \lambda\vert < \delta/2$ and all $i \geq i_0$.
Now, $g_f$ is holomorphic since it is the uniform limit of holomorphic functions.
\end{proof}

In the sequel, consider the space $Z=c_0$ with the norm defined by $\Vert x \Vert_Z = \Vert x \Vert_\infty + \left( \sum_{i = 1}^\infty (\frac{x(i)}{2^i})^2\right)^{1/2}$. Then $\Vert\cdot\Vert_Z$ and $\Vert\cdot\Vert_\infty$ are equivalent norms. Moreover,  $Z$ and $ Z''$ are strictly convex. The space $Z$ appears in classical counterexamples of norm attaining results (see for instance \cite{Lind}, \cite{AAGM}, \cite{Pa92})

\begin{proposition}\label{no vale BP vectorial}
The set $NA\mathcal{A}_u(c_0;Z'')$ is not dense in $\mathcal{A}_u(c_0;Z'')$.
\end{proposition}

\begin{proof}
Consider $Q: c_0 \rightarrow Z''$ defined by $Q(x)= x$. It is clear that $Q \in \mathcal{A}_u(c_0;Z'')$ and that $\Vert Q(e_n)\Vert_{Z''} \geq 1$ for all $n \in \mathbb N$. Fix $0 < \delta < 1$, take a norm attaining $f \in \mathcal{A}_u(c_0;Z'')$ and let $x_0 \in B_{c_0}$ be such that $\Vert f(x_0)\Vert = \Vert f\Vert$.
Since $x_0$ satisfies condition \eqref{propiedad extremal} for the fixed $\delta$ and some $n_0 \in \mathbb N$,
by Lemma \ref{lema extremal 2} the function
\begin{eqnarray*}
g_f : \left\{\vert \lambda\vert < \delta/2\right\} \rightarrow Z'' \\
g_f(\lambda) = f(x_0 + \lambda e_n)
\end{eqnarray*}
is holomorphic  for fixed $n \geq n_0$. Since $g_f$ attains its maximum at 0, it is constant and then $D^j g_f(0) =0$ for all $j\geq 1$. On the other hand, if we define $g_Q(\lambda) = Q(x_0 + \lambda e_n)$ then $g_Q$ is holomorphic and $D^1 g_Q(0)(\lambda) = \lambda Q(e_n)$. Now, by Cauchy inequalities we obtain
\begin{equation*}
1 \leq \Vert D^1 g_Q(0) - D^1 g_f(0)\Vert \leq \frac{1}{(\delta/2)} \sup_{\vert \lambda\vert < \delta/2}\Vert g_Q(\lambda) - g_f(\lambda)\Vert \leq \frac{2}{\delta} \Vert Q - f \Vert.
\end{equation*}
Hence, $Q$ cannot be approximated by norm attaining functions in $\mathcal{A}_u(c_0;Z'')$.
\end{proof}

It is worth noting that the argument in the previous proof does not work if we consider functions defined on $d_*(w,1)$ (instead of $c_0$) with values in a strictly convex Banach space. The reason is that, although any $x_0 \in B_{d_*(w,1)}$ satisfies condition \eqref{propiedad extremal}, we cannot fix $\delta$ independently of $f$. In fact, in this case $\delta$ depends on $x_0$ and can be arbitrarily small.

\bigskip

\subsection*{Counterexamples to strong versions of the Bishop-Phelps theorem}
We have already mentioned that we do not know wether the Bishop-Phelps theorem holds for $\mathcal{A}_u$ in the scalar-valued case. We show now that the strong versions of this theorem introduced in \cite{AcAlGaMa} which we studied in Section \ref{seccion Lindenstrauss} actually fail, while the corresponding strong versions of the Lindenstrauss theorem hold.
For this purpose, we state the next lemma which is analogous to Lemma ~\ref{derivadas igual a 0 caso polinomial}.

\begin{lemma} \label{derivadas igual a 0}
Let $X$ be a Banach sequence space and $W$ be strictly convex. Let $0<s<1$ and $f \in \mathcal{A}_u(X;W)$.
\begin{enumerate}
\item[\rm (i)] Fix $0<s_0<1$ and consider $x_0 \in s B_X$ for some $0< s<s_0$. Then $$\left\Vert \frac{D^j f(x_0)}{j!}\right\Vert \leq \frac{1}{(s_0 - s)^j} \Vert f\Vert_{s_0}$$ for all $j\geq 1$.
\item[\rm (ii)] Suppose that $f$ attains the $\Vert \cdot\Vert_s$-norm at an element $x_0 \in s B_X$ satisfying the following condition:
\begin{equation} \label{propiedad s-extremal}
\exists \  n_0 \in \mathbb{N} \  \text{\ and\ }\ \delta >0\  \text{\ such that\ }\quad  \Vert x_0 + \lambda e_n\Vert \leq s,\quad  \forall \ |\lambda| \le  \delta \text{\  and \  }  n \geq n_0.
\end{equation}
Then, $D^j f(x_0)(e_n) =0$ for all $j\geq 1$ and $n \geq n_0$.
\end{enumerate}
\end{lemma}

\begin{proof}
$(i)$ Fix $r= s_0 - \Vert x_0\Vert$ and $y \in B_X^{º}$, and consider the one variable holomorphic function
\begin{eqnarray*}
g_f : \{\vert\lambda\vert < 1\} &\longrightarrow& W \\
\lambda &\mapsto& f(x_0 + \lambda r y).
\end{eqnarray*}
By the Cauchy inequalities we have $$\left\Vert \frac{D^j g_f(0)}{j!}\right\Vert \leq \sup_{\vert \lambda\vert < 1} \Vert g_f(\lambda)\Vert \leq \Vert f\Vert_{s_0}$$ for all $j \geq 1$. Now, noting that $D^j g_f(0) = r^j D^j f(x_0)(y)$ we deduce $$r^j \left\Vert \frac{D^j f(x_0)(y)}{j!}\right\Vert \leq \Vert f\Vert_{s_0}$$ and since $y \in B_X^{º}$ was arbitrary, the desired statement follows.

$(ii)$ Since $\Vert x_0\Vert\leq s<1$ and $f$ is holomorphic in $B_X^{°}$, we can consider the series expansion of $f$ at $x_0$, $$f(x) = \sum_{j=0}^\infty \frac{D^j f(x_0)}{j!}(x-x_0).$$ Then the statement follows evaluating at $x=x_0 + \lambda e_n$ with $n \geq n_0$ and proceeding as in Lemma \ref{derivadas igual a 0 caso polinomial}
\end{proof}

We remark that, as expected, every element in the unit ball of $c_0$ or $d_*(w,1)$ satisfies condition~\eqref{propiedad s-extremal}.
The following result is the improvement of \cite[Corollary~4.5]{AcAlGaMa} mentioned above. Both Lemma \ref{derivadas igual a 0} and Proposition \ref{no vale s-BP} hold when we consider $H^\infty(B_{d_*(w,1)}^{º};W)$ instead of $\mathcal{A}_u(d_*(w,1);W)$.

\begin{proposition} \label{no vale s-BP}
Let $W$ be a strictly convex space and take an admissible sequence $w \in \ell_N$ for some $N\geq 2$. Given $0
 < s < s_0 \leq 1$, there exists an $N$-homogeneous polynomial that cannot be approximated in the
$\Vert\cdot\Vert_{s_0}$-norm (nor, in particular, in the $\|\cdot\|$-norm) by elements of
$\mathcal{A}_u(d_*(w,1);W)$ that attain the $\Vert \cdot\Vert_s$-norm.
\end{proposition}

\begin{proof}
Fix a norm-one element $z_0\in W$ and define $Q: d_*(w,1) \longrightarrow W$ by $$Q(x)=\left(\sum_{i=1}^\infty x(i)^N\right) z_0 .$$ Then $Q \in \mathcal{P}(^Nd_*(w,1);W)$ and its restriction to the ball  belongs to $\mathcal{A}_u(d_*(w,1);W)$. Take a function $f \in \mathcal{A}_u(d_*(w,1);W)$ that attains its $\Vert \cdot\Vert_s$-norm at an element $x_0 \in sB_{d_*(w,1)}$. By Lemma \ref{derivadas igual a 0} (ii), there exists $n_0 \in \mathbb N$ such that $D^j f(x_0)(e_n) =0$ for all $j\geq 1$ and $n \geq n_0$. On the other hand, we have $\frac{D^N Q(x_0)}{N!} =Q$ and hence $\Vert\frac{D^N Q(x_0)}{N!}(e_n)\Vert = 1$ for all $n \in \mathbb N$. For $n \geq n_0$, by Lemma \ref{derivadas igual a 0} (i) we have
\begin{eqnarray*}
1 = \left\Vert \frac{D^N Q(x_0)}{N!}(e_n) - \frac{D^N f(x_0)}{N!}(e_n)\right\Vert
\leq \frac{1}{(s_0 - s)^N} \Vert Q - f\Vert_{s_0}.
\end{eqnarray*}
Therefore, $Q$ cannot be approximated by a function $f \in \mathcal{A}_u(d_*(w,1);W)$ that attains its $\Vert \cdot\Vert_s$-norm. This gives the desired statement.
\end{proof}

Finally, we have the following equivalence in the spirit of Proposition \ref{equivalencias caso polinomial}. See also \cite[Proposition 2.6]{AAMo} where, with the same tools, a similar equivalence is given.
\begin{corollary}
Let $0<s<1$. The set of functions in $\mathcal{A}_u(d_*(w,1))$ attaining the $\Vert \cdot\Vert_s$-norm is $\Vert \cdot\Vert$-dense in $\mathcal{A}_u(d_*(w,1))$ if and only if
$w \notin \ell_N$ for all $N \in \mathbb N$.
\end{corollary}

\begin{proof}First note that we can proceed as in Corollary \ref{corolario equivalencias caso polinomial} to show that for $0<s<1$,  the set of $\Vert \cdot\Vert_s$-norm attaining polynomials in $\mathcal{P}_N(d_*(w,1))$ is $\Vert \cdot\Vert$-dense in $\mathcal{P}_N(d_*(w,1))$ if and only if $w \notin \ell_N$.
This implies, if $w \notin \ell_N$ for all $N \in \mathbb N$, that the set of $\Vert \cdot\Vert_s$-norm attaining polynomials is dense in the space $\mathcal{P}(d_*(w,1))$ of polynomials (of any degree). Then, given $g \in \mathcal{A}_u(d_*(w,1))$ and $\varepsilon >0$, we can take a polynomial $q$ such that $\Vert g - q\Vert< \varepsilon/2$, and then a $\Vert \cdot\Vert_s$-norm attaining polynomial $p$ such that $\Vert q - p\Vert< \varepsilon/2$.
This proves one implication, while the other follows from Proposition \ref{no vale s-BP}.
\end{proof}

As a consequence, taking $\mathcal{A}_u(d_*(w,1))$ with $w \in \ell_r$ for some $1<r<\infty$ we obtain, in the
scalar-valued case, the desired examples of spaces for which the strong version of the Bishop-Phelps theorem
fails, but the corresponding strong version of the Lindenstrauss theorem holds.
For arbitrary admissible sequences (which do not necessarily belong to some $\ell_r$), we can do the following in the vector-valued case. We take again $Z$ a renorming of $c_0$ such that its bidual $Z''$ is strictly convex, and consider $Q(x)=x$ which is well defined from $d_*(w,1)$ to $Z''$ regardless of $w$ belonging to some $\ell_r$. Moreover, $Q$ is well defined from $c_0$ to $Z''$ and the strong version of the Bishop-Phelps theorem fails also in this case since the Bishop-Phelps theorem fails according to the Proposition \ref{no vale BP vectorial}.

As in the polynomial case, we also have the previous equivalence when we consider holomorphic functions with values in $c_0$. Indeed, with slight modifications in the proof of \cite[Theorem~2.1]{ChoiKim96}, we have that if the strong version of Bishop-Phelps holds for $\mathcal{A}_u(d_*(w,1))$ and $Y$ has property $(\beta)$ then it holds for $\mathcal{A}_u(d_*(w,1);Y)$. For the other implication, the same polynomial $Q$ of Proposition \ref{equivalencias caso polinomial a valores en c0} works as a counterexample and the proof follows the same line, making use of Lemma \ref{derivadas igual a 0} to prove that
\begin{eqnarray*}
1 &=& \left\vert \frac{D^N (e_{m_0}^* \circ Q)(x_0)}{N!}(e_n) - \frac{D^N (e_{m_0}^* \circ f)(x_0)}{N!}(e_n)\right\vert
\leq \frac{1}{(1-s)^N} \Vert Q - f\Vert
\end{eqnarray*}
for any $\Vert \cdot\Vert_s$-norm attaining $f \in \mathcal{A}_u(d_*(w,1);c_0)$ and $n$ large enough. This gives the desired statement.

\end{document}